\documentclass[a4paper,egregdoesnotlikesansseriftitles,final]{scrartcl}
\usepackage[utf8]{inputenc}
\usepackage[british]{babel}
\usepackage[T1]{fontenc}
\usepackage{lmodern}
\usepackage[babel]{microtype}
\usepackage{amsfonts,amsthm,amssymb,amsmath}
\usepackage{thm-restate}
\usepackage{mathtools}
\usepackage{letterswitharrows}
\usepackage{color}
\usepackage{graphicx}
\usepackage{tikz}
\usepackage[ocgcolorlinks,unicode,bookmarks,pdfusetitle,linkcolor={red!60!black},citecolor={green!60!black},urlcolor={blue!60!black},bookmarksnumbered,bookmarksopen]{hyperref}
\usepackage[capitalize,english]{cleveref}
\usepackage{enumitem}
\usepackage[abbrev,msc-links]{amsrefs}
\usepackage{doi}
\renewcommand{\PrintDOI}[1]{\doi{#1}}

\makeatletter\renewcommand{\PrintDatePosted}[1]{\unskip\ (\print@date)}\makeatother
\newtheorem{THM}{Theorem}[section]
\newtheorem{LEM}[THM]{Lemma}

\theoremstyle{definition}
\newtheorem{EX}[THM]{Example}

\colorlet{colorA}{orange!70!white}
\colorlet{colorB}{cyan!60!black}
\renewcommand{\phi}{\varphi}

\newcommand{\join}{\lor}
\newcommand{\meet}{\land}
\newcommand{\sub}{\subseteq}
\newcommand{\sm}{\smallsetminus}

\newcommand{\menge}[1]{\left\{#1\right\}}
\newcommand{\abs}[1]{\left\lvert#1\right\rvert}
\newcommand{\tn}[1]{\textnormal{#1}}
\renewcommand{\S}{\mathcal{S}}
\renewcommand{\P}{\mathcal{P}}
\newcommand{\F}{\mathcal{F}}
\newcommand{\A}{\mathcal{A}}
\newcommand{\curlyle}{\preccurlyeq}
\renewcommand{\le}{\leqslant}
\renewcommand{\ge}{\geqslant}\renewcommand{\geq}{\geqslant}
\title{Trees of tangles in abstract separation~systems}
\date{20th January 2021}
\author{Christian Elbracht \and Jakob Kneip \and Maximilian Teegen}

\begin{document}
\maketitle
\begin{abstract}\noindent
    We prove canonical and non-canonical tree-of-tangles theorems for abstract separation systems that are merely structurally submodular.
    Our results imply all known tree-of-tangles theorems for graphs, matroids and abstract separation systems with submodular order functions,
    with greatly simplified and shortened proofs.
\end{abstract}

\section{Introduction}

Tangles were introduced in~\cite{GMX} by Robertson and Seymour as part of their graph minor project. Since then tangles have become one of the central objects of study in the field of graph minor theory. The general idea of tangles is to capture highly connected substructures in graphs not by listing the vertices and edges belonging to that structure but instead, indirectly, by listing for every low-order separation of the graph on which side of that separation the highly connected structure lies. In this way, the tangle `points towards' the dense structure rather than describing it explicitly. The paradigm shift brought about by Robertson and Seymour, then, was to view tangles themselves as highly cohesive objects, and working directly with these sets of oriented separations rather than with any substructures possibly captured by them.

The central theorem in the theory of tangles, which was established by Robertson and Seymour together with the notion of tangles, is the following:

\begin{restatable}[\cite{GMX}]{THM}{GMXthm}\label{thm:GMX}
    Every graph has a tree-decomposition displaying its maximal tangles.
\end{restatable}

\cref{thm:GMX} roughly says that the highly cohesive regions in a graph are arranged in a tree-like structure. The `maximal' in~\cref{thm:GMX} relates to the order of the tangles: the tree-decomposition found by~\cref{thm:GMX} displays the graph's tangles at every level of coarseness.

The original proof of~\cref{thm:GMX} by Robertson and Seymour in~\cite{GMX} is fairly involved and uses as tools multiple non-trivial results about separations in graphs, for instance, the existence of certain `tie-breaker' functions. Since then, the theory of tangles has moved on considerably, and shorter and more elementary proofs of~\cref{thm:GMX} have been found. The shortest proof to date is due to Carmesin \cites{ShortToT,Diestel5}, who utilises the fact that the separations needed for the tree-decomposition in~\cref{thm:GMX} behave well under appropriately defined joins and meets when taken to be of minimal order.

Carmesin, Diestel, Hundertmark, and Stein established the following strengthening of~\cref{thm:GMX}:

\begin{THM}[\cite{confing}]\label{thm:GMX-Canonical}
    Every graph has a canonical tree-decomposition displaying its maximal tangles.
\end{THM}

Here, `canonical' means that every automorphism of the graph acts on the decomposition tree.
In other words, \cref{thm:GMX-Canonical} uses only invariants of the graph---in particular no tie-breaker---to find the desired tree-decomposition.

A careful analysis of~\cite{confing}'s proof of \cref{thm:GMX-Canonical} conducted in~\cite{ProfilesNew} resulted in another shift of paradigm, similarly to the shift brought about by~\cite{GMX}. Much in the same way as tangles made it possible indirectly to capture substructures in graphs that were traditionally described more directly by sets of vertices or edges, and to treat them in a unified framework, it turned out that tangles themselves could be described, unlike in their definition given by Robertson and Seymour in \cite{GMX}, without reference to vertices or edges.

Indeed, the only information needed about a graph's tangles to prove \cref{thm:GMX-Canonical} is how its separations relate to each other, that is, which separations are nested or cross. Formally, the separations of the graph are turned into a poset together with an involution, and all subsequent tools and theorems in~\cite{ProfilesNew} are then formulated for these posets. This new notion of `abstract tangles' yielded not only a cleaner proof of \cref{thm:GMX-Canonical} in~\cite{ProfilesNew}, but also made the theory of tangles applicable to a wider range of combinatorial structures.

This novel way of working with so-called `abstract separation systems' is summarised in~\cite{AbstractSepSys}, and has yielded multiple generalisations and strengthenings of various theorems in the theory of tangles in both the finite (see~\cites{DualityAbstract, AbstractTangles, ProfilesNew, TreeSets, ProfileDuality, SeparationsOfSets, ToTviaTTD, CanonSubmod}) and the infinite (see~\cites{InfiniteSplinters, InfiniteTangles, ProfiniteSS, TreeSets, TreelikeSpaces}) setting. In this generalised framework tree-decompositions and tangles of graphs are generalised to nested sets of separations and `profiles', respectively. Working with abstract separation systems rather than with graphs makes many of the results in this new theory of tangles applicable to give notions of highly cohesive substructures in settings other than just graphs, such as in matroids or in image segmentation~(\cites{Duality2,ML}).

However one condition not expressed in terms of relations between the separations remained in use throughout the series of abstractions of~\cref{thm:GMX} implemented in~\cite{confing} and~\cite{ProfilesNew}: all of these works assumed that the separation systems of interest came with a submodular order function. Likewise, Carmesin's short proof of~\cref{thm:GMX} in~\cite{ShortToT} also leverages the fact that the order of separations of graphs is a submodular function.

This last non-structural aspect of tree-of-tangles theorems was disposed of in~\cite{AbstractTangles}: in that paper Diestel, Erde, and Weißauer replaced the order function with a purely structural notion of submodularity which can be expressed solely in terms of the lattice structure of the separation system. In doing so they established the most general and widely applicable variant of \cref{thm:GMX} to date:

\begin{restatable}[{\cite{AbstractTangles}*{Theorem~6}}]{THM}{DanielThm}\label{thm:Daniel}
    Let $ \vS $ be a structurally submodular separation system and $ \P $ a set of profiles of $ S $. Then $ S $ contains a nested set that distinguishes~$ \P $.
\end{restatable}

The relevant separation systems in graphs are all `structurally submodular', and therefore \cref{thm:Daniel} still applies to tangles in graphs. On the other hand there are separation systems that are structurally submodular but cannot be represented by graph separations~(\cite{SeparationsOfSets}). In particular \cref{thm:Daniel} can also be applied to separation systems which, unlike separations in graphs, do not come with any order function, such as arbitrary bipartitions of sets. This is a marked step forward from its predecessor \cref{thm:GMX}, whose original proof made heavy use of the order of particular separations.

However there is a trade-off involved in \cref{thm:Daniel}'s wider applicability: it does not imply \cref{thm:GMX}. Indeed, \cref{thm:Daniel} applied to a graph produces a tree-decomposition which displays just the graph's $ k $-tangles for arbitrary but fixed $k$. This is a significant weakening of \cref{thm:GMX}, which finds a decomposition displaying the graph's maximal tangles for all tangles orders simultaneously. Moreover, the tree-decomposition found by \cref{thm:GMX} is {\em efficient} in the sense that for every pair of tangles distinguished by the tree-decomposition, the separation in the decomposition distinguishing that pair of tangles is of the lowest possible order. Since \cref{thm:Daniel} makes only structural assumptions so as to be applicable to separation systems without any order function, \cref{thm:Daniel} cannot guarantee that the separations used by the nested set to distinguish a particular pair of tangles are of minimal order.

In this paper we bridge the gap between \cref{thm:GMX} and \cref{thm:Daniel} by establishing the following tree-of-tangles theorem which combines the upsides of both \cref{thm:GMX} and \cref{thm:Daniel}, i.e., which is as widely applicable as \cref{thm:Daniel} while still being as powerful and efficient as \cref{thm:GMX} when applied to tangles in graphs:

\begin{restatable}{THM}{SequencesThm}\label{thm:sequence}
    If $\mathcal{S} = (S_1,\dots,S_n)$ is a compatible sequence of structurally submodular separation systems inside a universe $U\!$, and $\P$ is a robust set of profiles in $\mathcal{S}$, then there is a nested set $ N $ of separations in $U$ which efficiently distinguishes all the distinguishable profiles in~$\P$\!.
\end{restatable}

\cref{thm:sequence} includes \cref{thm:Daniel} by taking a sequence of just one separation system, and it implies \cref{thm:GMX} by taking as separation systems $ S_k $ the sets of all separations of order $ <k $ of the given graph; the resulting nested set is the set of separations of the desired tree-decomposition.

The nested set $ N $ found by \cref{thm:sequence} has to contain for every pair of profiles in $\P$ a separation from that pair's `candidate set' of all those separations which (efficiently) distinguish that pair of profiles. Thus, to prove \cref{thm:sequence}, it suffices to show that one can pick an element from each of these `candidate sets' in a nested way.

As it turns out, there is a very simple and purely structural requirement of the way these `candidate sets' interact with each other which guarantees that it is possible to pick such a nested set:

\begin{restatable}[Splinter Lemma]{LEM}{splinterThm}\label{thm:splinter}
    Let $ U $ be a universe of separations and $ \A=(A_i)_{i\le n} $ a family of subsets of~$ U.$ If $ \A $ splinters, then we can pick an element $ a_i $ from each $ A_i $ so that $ \{a_1,\dots,a_n\} $ is nested.
\end{restatable}

\cref{thm:splinter}, in a sense, represents yet another step of abstraction in the theory of tangles: rather than working with the profiles themselves it works with the sets of separations distinguishing a given pair of profiles.

\cref{thm:splinter} not only implies \cref{thm:sequence}, but can also be used to prove \cref{thm:GMX} and \cref{thm:Daniel} directly. In fact \cref{thm:splinter} has a remarkably short proof (as we shall see in \cref{sec:splinter}), making it the shortest available proof of~\cref{thm:GMX} so far (see~\cref{sec:appl:GMX}). Moreover, the premise in~\cref{thm:splinter} is straightforward to check, and~\cref{thm:splinter} itself does not make reference to tangles or any specific implementations of them. As a result~\cref{thm:splinter} can be used in many different settings, implying variations of~\cref{thm:GMX} in a multitude of contexts. For example, after deriving in \cref{sec:applications_splinter} \cref{thm:GMX}, \cref{thm:Daniel}, and \cref{thm:sequence} from \cref{thm:splinter}, we use \cref{thm:splinter} to establish a new tree-of-tangles theorem in the setting of clique separations.

Since~\cref{thm:splinter} does not yield a canonical set of separations we cannot deduce~\cref{thm:GMX-Canonical} from it. We fix this in~\cref{sec:canonical} by establishing a version of~\cref{thm:splinter} which does give a canonical nested set, albeit under slightly stronger assumptions:

\begin{restatable}[Canonical Splinter Lemma]{LEM}{splinterThmcanonical}\label{thm:splinter_hierarchically}
    Let $ U $ be a universe of separations and let $ {\A=(A_i\mid i\in I)} $ be a collection of subsets of $ U $ that splinters hierarchically with respect to a partial order $ \curlyle $ on $ I $. Then there exists a nested set $ N=N(\A) $ meeting every $ A_i $ in~$ \A $.
    
    Moreover, $ N(\A) $ is canonical: if $ \phi $ is an isomorphism of separation systems between $ \bigcup_{i\in I}\vA_i $ and a subset of some universe $ U' $ such that the family $ \phi(\A)\coloneqq(\phi(A_i)\mid i\in I) $ splinters hierarchically with respect to $ \curlyle $, then $ N(\phi(\A))=\phi(N(\A)) $.
\end{restatable}

We make use of~\cref{thm:splinter_hierarchically} in~\cref{sec:applications_canonical} to obtain a new shortest proof of~\cref{thm:GMX-Canonical} and to extend~\cref{thm:GMX-Canonical} to two natural types of separations whose structural submodularity does not come from a submodular order function: clique separations, and circle separations.

All graph-theoretic terms and notation used but not explained here can be found in~\cite{Diestel5}.

\section{Separation systems and profiles}\label{sec:terminology}

Our terms and notation regarding separation systems follow those of~\cite{AbstractSepSys}, whereto we refer the reader for an in-depth introduction to abstract separation systems. In addition we use terms and definitions from~\cite{ProfilesNew}, which introduced the notion of profiles. For the convenience of the reader we shall offer a brief introduction to separation systems and profiles in this section. In addition to this we will work with a slightly more general notion of $ k $-profiles than that used in~\cite{ProfilesNew}; we shall indicate below where we deviate from the usual definitions.

A \emph{separation system} is a poset $\vS$ with an order-reversing involution ${}^*$. If an element of $ \vS $ is introduced as $ \vs $, we denote its inverse under the involution by $ \sv\coloneqq(\vs)^* $. The involution on $ \vS $ being order-reversing then means that $ \vr\le\vs $ if and only if $ \rv\ge\sv $ for all $ \vr,\vs\in\vS $. The elements of $ \vS $ are called (oriented) {\em separations}. In this paper all separation systems (and graphs) are assumed to be finite.

Given a separation $ \vs\in\vS $ we write $ s $ for the set $ \{\vs,\sv\} $, and call $ s $ the underlying {\em unoriented separation} of $ \vs $ (and of $ \sv $). Given an unoriented separation $ s $, we call $ \vs $ and $ \sv $ the two {\em orientations} of $ s $. If $ \vS $ is a separation system we write $ S $ for the set of all $ s $ with $ \vs\in\vS $. For clarity of writing we may informally use terms that are defined only for unoriented separations also when referring to their orientations, and vice-versa. Likewise, we use the term `separation' for both oriented and unoriented separations.

Two separations $ r $ and $ s $ are {\em nested} if they have orientations $ \vr $ and $ \vs $ such that $ \vr\le\vs $; otherwise $ r $ and $ s $ {\em cross}. A set of separations is nested if all its elements are pairwise nested.

A {\em universe of separations} is a separation system $ \vU $ in which the poset comes with join and meet operators $ \join $ and $ \meet $ which make it a lattice. Note that in a universe DeMorgan's law holds, i.e., $ (\vr\join\vs)^*=\rv\meet\sv $. Given two separations $ r $ and $ s $ their {\em corner separations} are the separations of the form $ \vr\join\vs $ (or their underlying unoriented separations), where $ \vr $ and $ \vs $ range over all possible orientations of $ r $ and $ s $.

A separation $\vr\in\vS$ is \emph{trivial} if $\vr < \vs$ and $\vr < \sv$ for some $\vs\in\vS$. An unoriented separation is trivial if it has a trivial orientation. If $ T\subseteq S $ is a set of non-trivial separations that are pairwise nested, then $ T $ is a {\em tree set}.

Separations $ \vs\in\vS $ with $ \vs\le\sv $ are called {\em small}. A set of oriented separations is {\em regular} if it contains no small separation, and a set of unoriented separations is {\em regular} if none of its elements has an orientation that is small. Note that if $ \vs $ is trivial, then $ \vs $ is small; the converse of this is not true in general.
\footnote{For a separation to be small, only its relation to its own inverse matters, while for triviality the other separations of the system are also relevant. As an example, if our separation system $\vS$ consists of just one separation $\vs$ and its inverse $\sv$ and we have $\vs \le \sv$, then $\vs$ is small but not trivial in $\vS$.}

For a universe $ \vU\!$, an {\em order function} on $ \vU $ is a function $ |\cdot|\colon\vU\to\mathbb{R}_{\ge 0} $ with $ |\vs|=|\sv| $ for all $ \vs\in\vU\!$. Given such an order function we let $ |s|\coloneqq|\vs| $. The universe $\vU$ is called \emph{submodular} if the order function is submodular, i.e., if
\[ |\vr| + |\vs| \geq |\vr\join\vs| + |\vr\meet\vs| \quad \forall \vr,\vs \in \vU. \]

A separation system $\vS \subseteq \vU$ inside a universe is called \emph{(structurally) submodular} if for any $\vr, \vs \in \vS$ at least one of $\vr\join\vs$ and $\vr\meet\vs$ also lies in $\vS\!$.

An \emph{orientation} of a set $ S $ of unoriented separations is a set consisting of precisely one orientation, $\vs$ or $\sv$, for every $s\in S$. An orientation $O$ is \emph{consistent} if there are no $\vr,\vs\in O$ with $\rv < \vs$. Two orientations $ O $ and $ O'$ (not necessarily of the same set~$ S $) are \emph{distinguished} by some separation $ s $ if they orient $s$ differently; if $ O $ and $ O' $ are distinguished by some $ s $ ,then they are {\em distinguishable}. In a context where an order function is given, $s$ distinguishes $O$ and $O'$ \emph{efficiently} if it is of minimal order among all separations that distinguish $O$ and $O'$.

Given a universe $ U $ and $ S\sub U $, a {\em profile of $ S $} is a consistent orientation $ P $ of $ S $ with the {\em profile property:}
\[ \forall \,\vr,\vs\in P \colon (\rv\meet\sv)\notin P \tag{P}\label{property:P} \]
If the universe of separations $ U $ has an order function, then for a fixed $ S\sub U $ we let $S_k\coloneqq\{s \in S \mid |s| < k\}$. A {\em $ k $-profile (in $ S $)}, then, is a profile of $ S_k $, and a {\em profile in $ S $} is a $ k $-profile for some integer~$ k $.

The above definitions of $ S_k $ and ($ k $-)profiles in $ S $ are a minor generalisation of the definitions used in \cite{ProfilesNew}. That paper exclusively studied profiles in $ U\! $, that is, always took $ S=U $ in these definitions. Indeed, if no $ S $ is specified we shall implicitly take $ S=U $ in these definitions; observe that in that case the separation systems $ S_k $ are submodular if the order function on $ U $ is submodular.

Traditionally, a separation of a graph $ G=(V,E) $ is a pair $ (A,B) $ of subsets of $ V $ with $ A\cup B=V $ such that there is no edge in $ G $ between $ A\sm B $ and $ B\sm A $. These separations of graphs are an instance of separation systems: the set $ \vS(G) $ of all separations of $ G $ consitutes a separation system when equipped with the involution $ (A,B)^*\coloneqq(B,A) $ and the partial order where $ (A,B)\le(C,D) $ if and only if $ A\sub C $ and $ B\supseteq D $. In fact, this partial order on $ \vS(G) $ is a lattice with $ (A,B)\join (C,D)=(A\cup C\,,\,B\cap D) $, and $ \vS(G) $ becomes a submodular universe by letting the order of a separation be $ \abs{(A,B)}\coloneqq\abs{A\cap B} $. The notion of profiles in separation systems then generalises and extends the notion of tangles in graphs: every $ k $-tangle in $ G $ is in fact a $ k $-profile in this universe $ \vS(G) $ of all separations of $ G $.

\section{The Splinter Lemma}\label{sec:splinter}

In this section we establish our first main result, \cref{thm:splinter}, from which we shall derive two previously known theorems as well as two new flavours of tree-of-tangles theorems in~\cref{sec:applications_splinter}. A cornerstone of the proofs of both \cref{thm:splinter} as well as of the two known results we shall derive from it is the following lemma:

\begin{LEM}[{\cite{ProfilesNew}*{Lemma~2.1}}]\label{lem:fish}
    Let $ U $ be a universe and $ r,s\in U $ two crossing separations. Every separation $ t $ that is nested with both $ r $ and $ s $ is also nested with all four corner separations of $ r $ and $ s $.
\end{LEM}

Typically, the proof of a tree-of-tangles theorem proceeds by starting with some set $ N $ of separations which distinguish some (or all) of the given tangles, and then repeatedly replacing elements $ r $ of $ N $ which cross some other element $ s $ of $ N $ with an appropriate corner separation of $ r $ and $ s $. \cref{lem:fish} is then used to show that each of these replacements makes $ N $ `more nested', and thus one eventually obtains a nested set $ N $ which distinguishes all the given tangles. (See for instance the proof of Theorem~4 of~\cite{AbstractTangles}.) Usually, in order to not reduce the set of tangles distinguished by $ N $, one has to take special care which corner separation of two crossing $ r $ and $ s $ in $ N $ one uses for replacement; this depends on the specific properties of the tangles at hand.

Our \cref{thm:splinter} seeks to eliminate this careful selection of corner separations for replacement: we will show that for a family $ (A_i)_{i\le n} $ of subsets of some universe~$ U $ we can find a nested set $ N $ meeting all the $ A_i $, provided that these sets $ A_i $ have one straightforward-to-check property. This lemma will imply many of the existing tree-of-tangles theorems by taking as sets $ A_i $ the sets of separations which distinguish the $ i $-th pair of tangles, and checking that the one assumption needed for \cref{thm:splinter} is met. Notably, \cref{thm:splinter} will make no reference at all to tangles or their specific properties. The proof of \cref{thm:splinter} will also utilise \cref{lem:fish}; however, the only assumption we need about the sets $ A_i $ is that for elements $ a_i $ and $ a_j $ of $ A_i $ and $ A_j $, respectively, one of their four corner separations lies in either $ A_i $ or $ A_j $. This condition will be easy to verify if one wants to deduce other tree-of-tangles theorems from \cref{thm:splinter}. In fact, the verification of this condition, which just asks for the existence of {\em some} corner separation of $ a_i $ and $ a_j $ in $ A_i\cup A_j $, will usually be much more straightforward than the hands-on arguments used in the original proofs of those tree-of-tangles theorems, which for their replacement arguments often need to prove the existence of a {\em specific }corner separation of $ a_i $ and $ a_j $. So let us define this condition formally.

Let $ U $ be a universe and $ \A=(A_i)_{i\le n} $ some family of non-empty subsets of~$ U $. We say that $ \A $ {\em splinters} if, for every crossing pair of $a_i \in A_i \sm A_j$ and $a_j \in A_j \sm A_i$, one of their four corner separations lies in $A_i \cup A_j$.


Observe that  a family $ (A_i)_{i\le n} $ of non-empty sets splinters if and only if for every pair $ a_i\in A_i $ and $ a_j\in A_j$ of separations, either some corner separation of $ a_i $ and~$ a_j $ lies in $ A_i\cup A_j $, or one of $ a_i $ and~$ a_j $ lies in $ A_i\cap A_j $.
This is, because if two separations $ a_i $ and $ a_j $ are nested, then these separations themselves are corner separations of the pair $ a_i $ and $ a_j $.

With this definition and \cref{lem:fish} we are already able to state and prove our first main result:

\splinterThm*
\begin{proof}
    We proceed by induction on $ n $. The assertion clearly holds for $ n=1 $. So suppose that $ n>1 $ and that the above assertion holds for all smaller values of~$ n $.
    
    Suppose first that we can find some $ a_i\in A_i $ such that $ a_i $ is nested with at least one element of $ A_j $ for each $ j\ne i $. Then the assertion holds: for $ j\ne i $ let~$ A_j' $ be the set of those elements of $ A_j $ that are nested with $ a_i $. Then $ (A_j'\mid j\ne i) $ is a family of non-empty sets which splinters by \cref{lem:fish}. Thus by the induction hypothesis we can pick a nested set $ \{a_j\in A_j'\mid j\ne i\} $, which together with $ a_i $ is the desired nested set.
    
    To conclude the proof it thus suffices to find an $ a_i $ as above. To this end, we apply the induction hypothesis to $ A_1,\dots, A_{n-1} $ to obtain a nested set consisting of some $ a_1,\dots, a_{n-1} $. Fix an arbitrary $ a_n\in A_n $. For all $ i<n $, if $ a_i $ itself or one of its corner separations with $ a_n $ lies in $ A_n $, this $ a_i $ is the desired separation for the above argument. Otherwise, for each $ i<n $, either $ a_n $ itself or one of its corner separations with $ a_i $ lies in~$ A_i $, in which case $ a_n $ is the desired separation for the above argument.
\end{proof}

We shall see in \cref{sec:applications_splinter} that this innocuous-looking lemma is actually strong enough to directly imply various existing tree-of-tangles theorems, including \cref{thm:GMX}.

\section{Applications of the Splinter Lemma}\label{sec:applications_splinter}

\subsection{A short proof of \texorpdfstring{\cref{thm:GMX}}{\autoref*{thm:GMX}}}\label{sec:appl:GMX}

As a first application of \cref{thm:splinter} let us give a short proof of \cref{thm:GMX}:

\GMXthm*

\pagebreak[1]
Let us first recall the relevant definitions of separations and tangles in graphs:

Let $ G=(V,E) $ be a graph. Then the set $ \vS=\vS(G) $ of all separations $ (A,B) $ of $ G $ is a separation system with involution $ (A,B)^*=(B,A) $ and the partial order in which $ (A,B)\le(C,D) $ if and only if $ A\sub C $ and $ B\supseteq D $. The {\em order} of a separation $ (A,B) $ of $ G $ is $ \abs{(A,B)}\coloneqq\abs{A\cap B} $. With this order function $ \vS $ becomes a submodular universe, where $ (A,B)\join(C,D)=(A\cup C\,,\,B\cap D) $. For $ (A,B)\in\vS $ we write $ \{A,B\} $ for the underlying unoriented separation. Furthermore we write $ \vS_k=\vS_k(G) $ for the set of all separations of $ G $ of order $ <k $.

A {$ k $-tangle in $ G $}, for an integer $ k $, is an orientation $ P $ of $ \vS_k $ with the {\em tangle property:}

\[ \forall\,\, (A_1,B_1),(A_2,B_2),(A_3,B_3)\in\tau \colon G[A_1]\cup G[A_2]\cup G[A_3]\ne G \tag{T}\label{property:T} \]

A $ k $-tangle of $ G $ is a {\em maximal} tangle of $ G $ if it is not the subset of some $ l $-tangle of $ G $ for some $ l>k $.

For a tree-decomposition $ (T,\mathcal{V}) $ of $ G $ and an edge $ xy $ of $ T $ let $ T_x $ and $ T_y $ denote the components of $ T-xy $ containing $ x $ and $ y $, respectively. Let $ U_x $ be the union of all bags $ V_z $ with $ z\in T_x $, and similarly let $ U_y $ be the union of all bags $ V_z $ with $ z\in T_y $. Then we call $ (U_x,U_y) $ the {\em separation induced by $ xy $}. The {\em set of separations induced by $ (T,\mathcal{V}) $} is the set of all separations induced by the edges of $ T $.

We say that a tree-decomposition $ (T,\mathcal{V}) $ of $ G $ {\em displays its maximal tangles} if the set of separations induced by $ (T,\mathcal{V}) $ efficiently distinguishes the set of all maximal tangles of $ G $.

If $ N $ is a nested set of separations of $ G $ it is straightforward to find a tree-decomposition of $ G $ whose set of induced separations is precisely $ N $ (see~\cites{GMX,TreeSets}). Therefore, in order to prove \cref{thm:GMX}, it suffices to find a nested set $ N $ of separations of $ G $ which efficiently distinguishes all maximal tangles of $ G $.

For every pair $ P,P' $ of distinct maximal tangles of $ G $ let
\[ A_{P,P'}\coloneqq\menge{\{A,B\}\in S(G)\mid \{A,B\}\tn{ efficiently distinguishes }P\tn{ and }P'}. \]
Since $ P $ and $ P' $ are not subsets of each other $ A_{P,P'} $ is a non-empty set.

Let $ \A $ be the family of all these sets $ A_{P,P'} $. A nested set of separations of $ G $ distinguishes all maximal tangles of $ G $ efficiently if and only if it contains an element of each $ A_{P,P'} $. Therefore the existence of such a set, and hence \cref{thm:GMX}, now follows directly from \cref{thm:splinter} once we show that $ \A $ splinters:

\begin{LEM}\label{lem:GMXsplinters}
    The family $ \A $ of all $ A_{P,P'} $ splinters.
\end{LEM}

\begin{proof}
    Let $P\neq P'$ and $Q\neq Q'$ be two pairs of distinct maximal tangles of $ G $ and let $\{A,B\}\in A_{P,P'}$ and $\{C,D\}\in A_{Q,Q'}$ be two crossing separations. We need to show that we have either $ \{A,B\}\in A_{Q,Q'} $ or $ \{C,D\}\in A_{P,P'} $, or that some corner separation of $ \{A,B\} $ and $ \{C,D\} $ lies in~$ {A_{P,P'}\cup A_{Q,Q'}} $.
    By switching their roles if necessary we may assume that $ \abs{(A,B)}\le\abs{(C,D)} $. 
    
    Since $ Q $ and $ Q' $ both orient $ (C,D) $, and $ \abs{(A,B)}\le\abs{(C,D)} $, both tangles also orient $ \{A,B\} $. If $ Q $ and $ Q' $ orient $ \{A,B\} $ differently, then $ \{A,B\} $ distinguishes them efficiently and hence lies in $ A_{Q,Q'} $. So suppose that $ Q $ and $ Q' $ contain the same orientation of $ \{A,B\} $, say, $ (A,B) $.
    
    By renaming them if necessary we may assume that $ (C,D)\in Q $ and $ (D,C)\in Q' $.
    
    Consider the corner separation $ (A\cup C\,,\,B\cap D) $ and suppose first that $ \abs{(A\cup C\,,\,B\cap D)}\le\abs{(C,D)} $. Then, by $ (A,B),(C,D)\in Q $ and the tangle property~(\ref{property:T}), $ Q $ must contain $ (A\cup C\,,\,B\cap D) $. On the other hand $ Q' $ must contain its inverse $ (B\cap D\,,\,A\cup C) $ since $ (D,C)\in Q' $. But then this corner separation efficiently distinguishes $ Q $ and $ Q' $ and hence lies in $ A_{Q,Q'} $.
    
    Thus we may suppose that $ \abs{(A\cup C\,,\,B\cap D)}\ge\abs{(C,D)} $. By a similar argument we may further suppose that $ \abs{(A\cup D\,,\,B\cap C)}\ge\abs{(C,D)} $. Submodularity then yields $ \abs{(A\cap C\,,\,B\cup D)},\abs{(A\cap D\,,\,B\cup C)}\le\abs{(A,B)} $.
    
    By switching the roles of $ P $ and $ P' $ if necessary we may assume that $ (A,B)\in P $ and $ (B,A)\in P' $. Then by the above inequality $ P $ must contain both $ (A\cap C\,,\,B\cup D) $ and $ (A\cap D\,,\,B\cup C) $, since it cannot contain either of their inverses due to  $ (A,B)\in P $ and the tangle property~(\ref{property:T}). However, due to $ (B,A)\in P' $ and the tangle property~(\ref{property:T}), $ P' $ cannot contain both of $ (A\cap C\,,\,B\cup D) $ and $ (A\cap D\,,\,B\cup C) $. In must therefore contain the inverse of at least one of these corner separations, which then efficiently distinguishes $ P $ and $ P' $ and hence lies in $ A_{P,P'} $.
\end{proof}

\subsection{Abstract tangles in structurally submodular separation systems}\label{sec:appl:Daniel}

The most general, or most widely applicable, tree-of-tangles theorem published so far, in the sense of having the weakest premise, is the following:

\DanielThm*

The price to pay in \cref{thm:Daniel} for having the mildest set of requirements is that its assertion is also among the weakest of all tree-of-tangles theorems. For graphs, \cref{thm:Daniel} implies only that \emph{for any fixed $ k $} every graph has a tree decomposition displaying its $ k $-tangles. This is a much weaker statement than \cref{thm:GMX}, which finds a tree-decomposition displaying the maximal $ k $-tangles of that graph for all values of $ k $ simultaneously.

Let us show how to derive \cref{thm:Daniel} from \cref{thm:splinter}. For this, let $\P$ be a set of profiles of a submodular separation system $S$, and for distinct $P$ and $P'$ in $\P$ let \[
    A_{P,P'}:=\{ s\in S \mid s \text{ distinguishes }P\tn{ and }P'\}.
\]
For proving \cref{thm:Daniel} it suffices to show that the family $A_\P = (A_{P,P'}\mid P\ne P'\in \P)$ splinters:

\begin{LEM}\label{lem:Danielsplinter}
Given a set $\P$ of profiles of a submodular separation system $\vS$, the family $\A_\P = (A_{P,P'}\mid P\neq P'\in \P)$ splinters.
\end{LEM}

\begin{proof}
Let $P\neq P'$ and $Q\neq Q'$ be two pairs of profiles in $\P$ and let $r\in A_{P,P'}$ and $s\in A_{Q,Q'}$ be two distinct separations. We need to show that we have either $ r\in A_{Q,Q'} $ or $ s\in A_{P,P'} $, or that some corner separation of $ r $ and $ s $ lies in~$ {A_{P,P'}\cup A_{Q,Q'}} $. If $ r $ and $ s $ are nested, then they themselves are corner separations of $ r $ and $ s $ and there is nothing to show, so let us suppose that $ r $ and $ s $ cross.

Both $ r $ and $ s $ are oriented by all four profiles $ P,P',Q $, and $ Q' $. If $ r $ distinguishes $ Q $ and $ Q' $, or if $ s $ distinguishes $ P $ and $ P' $, we are done; so suppose that there are orientations $ \vr $ and $ \vs $ of $ r $ and $ s $ with $ \vr\in Q\cap Q' $ and $ \vs\in P\cap P' $. By possibly switching the roles of $ P $ and $ P' $, or of $ Q $ and $ Q' $, we may further assume that $ \rv\in P $ and $ \vr\in P' $ as well as $ \sv\in Q $ and $ \vs\in Q' $.

The submodularity of $ S $ implies that at least one of the two corner separations $ \rv\join\vs $ and $ \vr\join\sv $ lies in $ \vS $. We will only treat the case that $ (\rv\join\vs)\in\vS $; the other case is symmetrical.

From the assumption that $ r $ and $ s $ cross it follows that $ \rv\join\vs $ is distinct from $ r $ and $ s $ as an unoriented separation. Therefore, by $ \vr\in P' $ and consistency, $ P' $ cannot contain $ \rv\join\vs $ and hence has to contain its inverse $ \vr\meet\sv $. On the other hand, by $ \rv,\vs\in P $ and the profile property~(\ref{property:P}), $ P $ cannot contain the inverse of $ \rv\join\vs $ and thus must contain $ \rv\join\vs $. Now $ \rv\join\vs $ distinguishes $ P $ and $ P' $ and is therefore the desired corner separation in $ A_{P,P'} $.
\end{proof}

Let us now deduce \cref{thm:Daniel} from \cref{thm:splinter}.

\begin{proof}[Proof of \cref{thm:Daniel}]
 Let $\P$ be a set of profiles of $S$. By \cref{lem:Danielsplinter} the collection $(A_{P,P'}\mid P\ne P'\in\mathcal{P})$ of subsets of~$S$ splinters. Each of the $A_{P,P'}$ is non-empty as $P$ and $P'$ are distinct profiles of $S$. Thus, by \cref{thm:splinter}, we can pick one element from each $A_{P,P'}$ so that the set $N$ of all these elements is a nested set of separations. It is then clear that $ N $ distinguishes all the profiles in $ \P $.
\end{proof}

The above way of using \cref{thm:splinter} to prove a tree-of-tangles theorem is archetypical, and we will use the strategy from this section as a blueprint for the applications of \cref{thm:splinter} in the following sections.

\subsection{Profiles of separations}\label{sec:appl:profiles_noncanonical}

\cref{thm:Daniel} from the previous section implied that every graph has, for any fixed integer~$ k $, a tree-decomposition which displays its $ k $-tangles. However, Robertson's and Seymour's~\cref{thm:GMX} shows that every graph has a tree-decomposition which displays all its {\em maximal} tangles, i.e., which distinguishes all its distinguishable tangles for all values of $ k $ simultaneously, not just for some fixed value of~$ k $. Therefore \cref{thm:Daniel} does not imply \cref{thm:GMX}.

Moreover, since \cref{thm:Daniel} does not assume that the universe $ U $ it is applied to comes with an order function, \cref{thm:Daniel} cannot say anything about the order of the separations used in the nested set to distinguish all the profiles. If the universe $ U $, as for instance in a graph, {\em does} come with a submodular order function, one might ask for a nested set which not only distinguishes all the profiles given, but one which does so {\em efficiently}, i.e., which contains for every pair $ P,P' $ of profiles a separation of minimal order among all the separations in $ U $ which distinguish $ P $ and $ P' $.

The following theorem satisfies both of the requirements above, and is the strongest tree-of-tangles theorem known so far:

\begin{restatable}[Canonical tangle-tree theorem for separation universes {\cite{ProfilesNew}*{Theorem~3.6}}]{THM}{profilesCanonToT}\label{thm:Profiles_canon}
Let $U=(\vU,\le,{}^*,\join,\meet,|\cdot|)$ be a submodular universe of separations. Then for
every robust set $\P$ of profiles in $U$ there is a nested set $T = T(\P)\subseteq U$ of
separations such that:
\begin{enumerate}[label=(\roman*)]
\item every two profiles in $\P$ are efficiently distinguished by some separation in $T$;
\item every separation in $T$ efficiently distinguishes a pair of profiles in $\P$;
\item for every automorphism $\alpha$ of $\vU$ we have $T(\P^\alpha) = T(\P)^\alpha$; \hfill \emph{(canonicity)}
\item if all the profiles in $\P$ are regular, then $T$ is a regular tree set.
\end{enumerate}
\end{restatable}

Since the definition of robustness of (a set of) profiles is rather involved we do not repeat it here. In the following proofs robustness will be used only in one place; therefore we shall use it there as a black box and refer the reader to~\cite{ProfilesNew} for the full definition.

Since every $ k $-tangle of a graph is robust~(\cite{ProfilesNew}), \cref{thm:Profiles_canon} indeed implies \cref{thm:GMX} of Robertson and Seymour that every graph has a tree-decomposition displaying its maximal tangles (see \cite{ProfilesNew}*{Section 4.1} for more on building tree-decompositions from nested sets of separations, and how \cref{thm:Profiles_canon} implies \cref{thm:GMX}). Moreover, \cref{thm:Profiles_canon} improves upon \cref{thm:GMX} by finding a {\em canonical} such tree-decomposition, i.e., one which is preserved by automorphisms of the graph. Since \cref{thm:splinter} does not guarantee any kind of canonicity, we are not able to deduce the full \cref{thm:Profiles_canon} from \cref{thm:splinter}; however, using \cref{thm:splinter} we will be able to find a nested set $ T\sub U $ with the properties (i), (ii) and (iv). We shall refer to this as the \emph{non-canonical \cref{thm:Profiles_canon}}. (In \cref{sec:canonical} we shall prove a version of \cref{thm:splinter} which implies \cref{thm:Profiles_canon} in full.)

Our strategy will largely be the same as in \cref{sec:appl:Daniel}. For a robust set $ \P $ of profiles in a submodular universe $ U $ we define for every pair $ P,P' $ of distinct profiles in $ \P $ the set
\[ A_{P,P'}:=\menge{a\in U\mid a\tn{ distinguishes }P\tn{ and }P'\tn{ efficiently}}. \]
Let $ \A_\P $ be the family $ (A_{P,P'} \mid P\neq P' \in \P) $. The only lemma we need in order to apply \cref{thm:splinter} is the following:

\begin{LEM}\label{lem:eff_splinter}
    For a robust set $\P$ of profiles in $U$ the family $\A_\P$ of the sets $ A_{P,P'} $ splinters.
\end{LEM}

\begin{proof}
    Let $P,P'$ and $Q,Q'$ be two pairs of distinguishable profiles in $\P$ and let $r\in A_{P,P'}$ and $s\in A_{Q,Q'}$ be two crossing separations. We need to show that we have either $ r\in A_{Q,Q'} $ or $ s\in A_{P,P'} $, or that some corner separation of $ r $ and $ s $ lies in~$ {A_{P,P'}\cup A_{Q,Q'}} $. By switching their roles if necessary we may assume that $ \abs{r}\le\abs{s} $.
    
    Since $ Q $ orients all separations in $ U $ of order at most the order of $ s $, $ Q $ contains some orientation $ \vr $ of $ r $. Similarly $ Q' $ contains some orientation of $ r $: if $ \rv\in Q' $ then $ r $ distinguishes $ Q $ and $ Q' $, and by $ \abs{r}\le\abs{s} $ it does so efficiently, giving $ r\in A_{Q,Q'} $. So suppose that $ \vr\in Q' $.
    
    If either one of the two corner separations $ \vr\join\vs $ and $ \vr\join\sv $ has order at most the order of $ s $, then that corner separation would distinguish $ Q $ and $ Q' $ by the profile property. In particular, that corner separation would do so efficiently and hence lie in $ A_{Q,Q'} $. Thus we may assume that both of these corner separations have order strictly larger than the order of $ s $.
    
    The submodularity of $ U $ now implies that both of the other two corner separations, that is, $ \vr\meet\vs $ and $ \vr\meet\sv $, have order strictly less than the order of $ r $. Therefore both $ P $ and $ P' $ orient both of these corner separations. By possibly switching the roles of $ P $ and $ P' $ we may assume that $ \rv\in P $ and $ \vr\in P' $. Then $ P' $ contains both $ \vr\meet\vs $ and $ \vr\meet\sv $ due to consistency, since both of these corner separations are distinct from $ r $ as unoriented separations by the assumption that $ r $ and $ s $ cross.
    
    But now the assumption that $ r $ distinguishes $ P $ and $ P' $ efficiently implies that neither of the two corner separations $ \vr\meet\vs $ and $ \vr\meet\sv $ can distinguish $ P $ and $ P' $, since the corner separations have strictly lower order than $ r $. Therefore $ P $ contains $ \vr\meet\vs $ and $ \vr\meet\sv $ as well. However, by $ \rv\in P $, this contradicts the robustness of $ P $, which forbids exactly this configuration.
\end{proof}

Let us now deduce the non-canonical \cref{thm:Profiles_canon} from \cref{thm:splinter}:

\begin{proof}[Proof of the non-canonical \cref{thm:Profiles_canon}]\hypertarget{prfprofilesnoncanon}{}
By \cref{lem:eff_splinter} the collection $\A_\P$ of the sets $ A_{P,P'} $ splinters. Thus by \cref{thm:splinter} we can pick an element from each set $ A_{P,P'} $ in $ \A_\P $ in such a way that the set $ T $ of these elements is nested. Let us show that this set $ T $ is as claimed.

For (i), let $ P $ and $ P' $ be two profiles in $ \P $. As $ T $ meets the set $ A_{P,P'} $, some element of $ T $ distinguishes $ P $ and $ P' $ by definition of $ A_{P,P'} $.

For (ii), observe that every element of $ T $ lies in some $ A_{P,P'} $ and hence distinguishes a pair of profiles in $ \P $ efficiently.

Finally, (iv) follows from the fact that all sets $ A_{P,P'} $ in $ \A_P $ are regular if every profiles in $ \P $ is regular, which implies that $ T $ is a regular tree set in that case.
\end{proof}

\subsection{Sequences of submodular separation systems}\label{sec:appl:sequences}

Let us, once more, compare \cref{thm:Daniel} and \cref{thm:Profiles_canon}. The first of these has the advantage that it does not depend on any order function and thus applies to a wider class of universes of separations; on the other hand, for those universes that do have an order function, the latter theorem is much more flexible and powerful, since it not only distinguishes all distinguishable profiles across all orders simultaneously, but also does so efficiently.

Our aim in this section is to establish \cref{thm:sequence} which combines the advantages of both \cref{thm:Daniel} and \cref{thm:Profiles_canon} (without canonicity), i.e., which is not dependent on the existence of some order function, but which is as powerful and efficient as \cref{thm:Profiles_canon} if such an order function does exist.

Concretely, we shall answer the following question, which inspired this research:
\begin{center}\em
    If $S_1 \sub S_2 \sub\ldots\sub S_n$ is an ascending sequence of structurally submodular separations systems exhausting a universe of separations $U$, does there exist a nested set of separations 
    which efficiently distinguishes all the\\ maximal profiles in~$U$? 
\end{center}
Let us substantiate this question with rigorous definitions of the terms involved.

We call a sequence $S_1\sub S_2\sub \ldots\sub S_n\sub U$ of submodular separation systems in a universe $U$ \emph{compatible} if for all pairs $s_i\in S_i$ and $s_j\in S_j$ with $ i\le j $, either $S_i$ contains at least two corner separations of $ s_i $ and $ s_j $, or $ S_j $ contains at least three corner separations of $ s_i $ and~$ s_j $.

Observe that if $ U $ comes with a submodular order function $ |\cdot| $ and the $ S_i $ are defined as in~\cref{sec:appl:profiles_noncanonical}, i.e., if $ S_i $ is the set of all separations in $ U $ of order $ <i $, then the sequence $ S_1\sub S_2\sub \ldots \sub S_n\sub U $ is a compatible sequence of submodular separation systems.

A \emph{profile in} $\mathcal{S} = (S_1, \dots, S_n)$ is a profile of one of the $S_i$.

A separation $s\in S_n$ \emph{distinguishes} two profiles $P$ and $Q$ in $\mathcal{S}$ if there are orientations of $s$ such that $\vs\in P $ and $ \sv\in Q$.
The separation $ s $ distinguishes $ P $ and $ Q $ {\em efficiently} if $ s\in S_i $ for every $ S_i $ which contains a separation that distinguishes $ P $ and~$ Q $.

Note once more that, as above, these notions of profiles and efficient distinguishers coincide with their usual definitions as given in~\cref{sec:appl:profiles_noncanonical} if $ U $ has a submodular order function and the $ S_i $ are the subsets of $ U $ containing all separations of order~$ {<i} $.

We also require a structural formulation of the concept of robustness from~\cite{ProfilesNew}: \\
A set $\P$ of profiles in $\mathcal{S}$ is \emph{robust} if for all $P, Q, Q'\in \P$ the following holds:
for every $ \vr\in Q\cap Q' $ with $ \rv\in P $ and every $ s $ which distinguishes $ Q $ and $ Q' $ efficiently, if $ s\in S_j $, then there is an orientation $ \vs $ of $ s $ such that either $ (\rv\join\vs)\in P $ or $ (\vr\join\vs)\in\vS_j $.

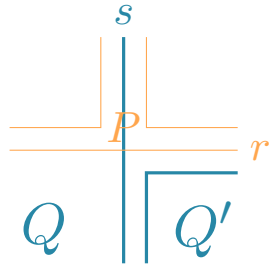
\begin{figure}[h]\centering
    \begin{tikzpicture}[scale=.15]
    \draw[very thick,colorB] (  0,-10) -- (  0, 10);
    \node[colorB, scale=1.5] at (0,12) {$s$};
    \node[colorB, scale=2] at ( 7,-7) {$Q'$};
    \node[colorB, scale=2] at (-7,-7) {$Q$};
    \draw[colorA] (-10,  0) -- ( 10,  0);
    \node[colorA, scale=1.5] at (12, 0) {$r$};
    \node[colorA, scale=1.5] at (0, 2) {$P$};
    \draw[very thick, colorB] (2, -10) |- (10, -2);
    \draw[colorA] (2, 10) |- (10, 2);
    \draw[colorA] (-2, 10) |- (-10, 2);
    \end{tikzpicture}
    \caption{Robustness.}
\end{figure}

With the above definitions we are now able to formally state and prove \cref{thm:sequence}, which includes both \cref{thm:Daniel} and the non-canonical \cref{thm:Profiles_canon} (and hence \cref{thm:GMX}) as special cases:

{
    \SequencesThm*
}

Since the proof of \cref{thm:sequence} runs along very similar lines as the \hyperlink{prfprofilesnoncanon}{proof} of \cref{thm:Profiles_canon} in the previous section we only sktech it here:

\begin{proof}[Sketch of proof.]
    For every pair $ P,P' $ of distinguishable profiles in $ \P $ let $ A_{P,P'} $ be the set of all $ s\in S_n $ that distinguish $ P $ and $ P' $ efficiently. The assertion of \cref{thm:sequence} follows directly from \cref{thm:splinter} if we can show that the family $ \A $ of these sets $ A_{P,P'} $ splinters.
    
    So let $ r\in A_{P,P'} $ and $ s\in A_{Q,Q'} $ be given. If $ r $ and $ s $ are nested there is nothing to show, so suppose they cross. Let $ i $ and $ j $ be minimal integers such that $ r\in S_i $ and $ s\in S_j $; we may assume without loss of generality that~$ i\le j $.
    
    If $ r $ distinguishes $ Q $ and $ Q' $, then $ r\in A_{Q,Q'} $, so suppose not, that is, suppose that some orientation $ \vr $ of $ r $ lies in both $ Q $ and~$ Q' $.
    
    If one of the two corner separations $ \vr\join\vs $ and $ \vr\join\sv $ lies in $ \vS_j $, then that separation distinguishes $ Q $ and $ Q' $ by consistency and the profile property and hence would lie in~$ A_{Q,Q'} $. So we may suppose that neither of these two corner separations lies in~$ \vS_j $. The compatability of $ \S $ then implies that both of the other two corner separations, $ \rv\join\vs $ and $ \rv\join\sv $, lie in~$ S_i $.
    
    By possibly switching the roles of $ P $ and $ P' $ we may assume that $ \vr\in P' $ and $ \rv\in P $. Then the robustness of $ \P $ implies that $ P $ contains either $ \rv\join\vs $ or $ \rv\join\sv $. This corner separation then lies in $ A_{P,P'} $ due to the consistency of $ P' $.
\end{proof}

\cref{thm:sequence} directly implies both \cref{thm:Daniel} and the non-canonical \cref{thm:Profiles_canon}: for the first theorem, consider the singleton sequence $ S_1=S $; and for the latter, take as $ S_i $ the set of all separations of order $ {<i} $ and let $ n $ be large enough that $ S_n=U $.

\subsection{Clique-separations in finite graphs}\label{sec:appl:cliques_noncanonical}

For a finite graph $ G $ a separation $(A,B)$ of $G$ is a \emph{clique separation} if the induced subgraph $G[A\cap B]$ is a complete graph. Clique separations in graphs have been studied by various people over the course of the last century~\cites{wagner,halin}.
More recently clique separations have received quite some attention in theoretical computer science (see for instance~\cites{LEIMER1993,berry2010,coudert:clique_decomp}) following Tarjan's work~\cite{tarjan1985} on their algorithmic aspects.

In~\cite{AbstractTangles} it was shown that the theory of submodular separation systems can be applied to clique separations of finite graphs to deduce the existence of certain nested distinguishing sets. Using \cref{thm:splinter} directly instead of \cref{thm:Daniel}, we are able to obtain a stronger result than the one given in~\cite{AbstractTangles}, much in the same way that \cref{thm:sequence} improves upon \cref{thm:Daniel}.

For this section let $ G=(V,E) $ be a finite graph, $ \vU=\vU(G) $ the universe of separations of $ G $, and $\vS=\vS(G)\subseteq \vU$ the separation system of all clique separations of~$G$. Consequently $ \vS_k=\vS_k(G) $ is the set of all clique-separations in $G$ of order less than $k$, i.e., the set of all $(A,B)\in \vS$ such that  $|A\cap B|<k$.

It was shown in \cite{AbstractTangles}*{Lemma~17} that $S$ is a submodular separation system. Following their proof, we can show that in fact every $S_k\sub S$ is a submodular separation system, and that these extend each other in a way similar to the ordinary $S_k$ of $G$:

\begin{LEM}\label{lem:clique_corners}
    Let $ r $ and $ s $ be two crossing clique separations with $ \abs{r}\le\abs{s} $. Then there are orientations $ \vr $ and $ \vs $ of $ r $ and $ s $ such that $ (\rv\meet\sv),(\rv\meet\vs), $ and $ (\vr\meet\sv) $ are clique separations with $ \abs{\rv\meet\sv}\le\abs{r}$ and $\abs{\rv\meet\vs}\le\abs{r} $ as well as $ \abs{\vr\meet\sv}\le\abs{s} $. Moreover, if $ \abs{\rv\meet\sv}=\abs{r}=\abs{s} $, then $ (\vr\meet\vs) $ is also a clique separation with $ \abs{\vr\meet\vs}\le\abs{r} $.
\end{LEM}

\begin{proof}
Let $s=\{A,B\}$ and $t=\{C,D\}$ be two crossing clique separations of $ G $ with $ \abs{r}\le\abs{s} $. Since $ C\cap D $ is a separator of $ G $, and all vertices in $ A\cap B $ are pairwise adjacent, $ A\cap B $ must be a subset of either $ C $ or $ D $. Similarly $ C\cap D $ must be a subset of either $ A $ or $ B $. By renaming the sets if necessary we may assume that $ A\cap B\sub C $ and $ C\cap D\sub A $. We orient $ r $ as $ \vr=(A,B) $ and $ s $ as $ \vs=(C,D) $; let us show that these orientations are as claimed.

Observe first that the separators of both $ (\rv\meet\sv) $ and $ (\rv\meet\vs) $ are subsets of $ A\cap B $, showing that these are clique separations of order at most $ \abs{r}=\abs{A\cap B} $. Similarly, the separator of the corner separation $ (\vr\meet\sv) $ is a subset of $ C\cap D $, and hence $ (\vr\meet\sv) $ is a clique separation of order at most $ \abs{s}=\abs{C\cap D} $.

Finally, suppose that $ \abs{\rv\meet\sv}=\abs{r}=\abs{s} $. Then, since the separator of $ (\rv\meet\sv) $ is a subset of both $ A\cap B $ and of $ C\cap D $, this separator must in fact be equal to both $ A\cap B $ and $ C\cap D $. Consequently the separator of $ (\vr\meet\vs) $ also equals $ A\cap B=C\cap D $, which shows that $ (\vr\meet\vs) $ is a clique separation of order at most $ r $.
\end{proof}

We can now consider profiles in $ G $ with respect to these separation systems. We will use the same notion of profiles as was introduced in~\cite{ProfilesNew}: a \emph{profile} $P$ of order $k$ is a consistent orientation of $S_k$ satisfying the profile property \[
    \forall \,\vr,\vs\in P \colon (\rv\meet\sv)\notin P\,. \tag{P}
\]
Every hole in $G$ (i.e.\ an induced cycle of length at least $4$) defines a profile $ P $ of order $ \abs{V} $ in $ G $ by letting $ P $ contain a separation $ (A,B)\in\vS $ of order less than $ \abs{V} $ if and only if that hole is contained in~$ G[B] $. In an analogous way every clique of size $k$ defines a profile of order $k$ in $ G $.
Let us denote by $\P_k$ the set of all profiles of order $k$. 

As usual, given two distinguishable profiles $P$ and $P'$, let \[A_{P,P'}:=\{a\in S \mid a \text{  distinguishes $P,P'$ efficiently}\}.\]

We will show that the collection of these $A_{P,P'}$ splinters.

\begin{LEM}\label{lem:clique_eff_splinter}
    For any set $\P$ of profiles the collection \[ {(A_{P,P'}\mid P,P'\tn{ distinguishable profiles})} \] splinters.
\end{LEM}

\begin{proof}
    Let $P,P'$ and $Q,Q'$ be two pairs of distinguishable profiles in $\P$ and let $r\in A_{P,P'}$ and $s\in A_{Q,Q'}$ be two distinct separations. We need to show that we have either $ r\in A_{Q,Q'} $ or $ s\in A_{P,P'} $, or that some corner separation of $ r $ and $ s $ lies in~$ {A_{P,P'}\cup A_{Q,Q'}} $. If $ r $ and $ s $ are nested, then the latter is immediate, so suppose that $ r $ and $ s $ cross. By switching their roles if necessary we may further assume that $\abs{r}\le\abs{s} $.
    
    Since $ Q $ orients $ s $, and $ \abs{r}\le\abs{s} $, the profile $ Q $ contains some orientation $ \vr $ of $ r $. Similarly $ Q' $ contains some orientation of $ r $. If $ \rv\in Q' $, then $ r $ distinguishes $ Q $ and $ Q' $, and by $ \abs{r}\le\abs{s} $ it does so efficiently, giving $ r\in A_{Q,Q'} $. So suppose that $ \vr\in Q' $.
    
    By~\cref{lem:clique_corners} at least three of the corner separations of $ r $ and $ s $ are clique separations of order at most $ \abs{s} $. Thus at least one of $ (\rv\meet\sv) $ and $ (\rv\meet\vs) $ is a clique separation of order at most $ \abs{s} $. This corner separation then distinguishes $ Q $ and $ Q' $ by the profile property, and in fact it does so efficiently, since its order is at most $ \abs{s} $, yielding the desired corner separation in $ A_{Q,Q'} $.
\end{proof}

It is now straightforward to use~\cref{thm:splinter} to obtain the following theorem:

\begin{THM}\label{thm:cliques}
There is a nested set of separations which efficiently distinguishes all the distinguishable profiles in $\bigcup_{i=1}^n \P_i$.
\end{THM}

\begin{proof}
By \cref{lem:clique_eff_splinter}, we can apply \cref{thm:splinter} to \[ {(A_{P,P'}\mid P,P'\tn{ distinguishable profiles})}, \] resulting in the claimed nested subset.
\end{proof}
In particular, for any two holes, a hole and a clique, or two cliques if there is a clique separation which distinguishes them, then our nested set contains one such separation of minimal order. As usual, such a nested set can be transformed into a tree-decomposition of $ G $ (see~\cite{TreeSets} for details). Thus  $G$ admits a tree-decomposition whose adhesion sets are cliques and which efficiently distinguishes all the holes and cliques distinguishable by clique separations in $G$. Such a decomposition is similar to, but not exactly the same as, the decomposition constructed by R.~E.~Tarjan in~\cite{tarjan1985}.

We will see in \cref{sec:appl:cliques_canonical} that such a decomposition can in fact be chosen canonically, i.e., to be invariant under automorphisms of $G$.

\section{Canonical Splinter Lemma}\label{sec:canonical}

As we saw in the previous section, \cref{thm:splinter} is already strong enough to imply most of \cref{thm:Profiles_canon}, but crucially does not guarantee the canonicity asserted in (iii). In this section we wish to prove a version of \cref{thm:splinter} using a stronger set of assumptions from which we can deduce \cref{thm:Profiles_canon} in full: we want to find, for a family $ \A=(A_i\mid i\in I) $ of subsets of some universe $ U $, a nested set $ N=N(\A) $ meeting all the $ A_i $ that is {\em canonical}, i.e., which only depends on invariants of $ \A $. More formally, we want to find $ N=N(\A) $ in such a way that if $ \A'=(A_i'\mid i\in I) $ is another family of subsets of some other universe $ U' $ that also meets the assumptions of our theorem, and $ \phi $ is an isomorphism of separation systems between $ \bigcup_{i\in I}A_i $ and $ \bigcup_{i\in I}A_i' $ with $ \phi(A_i)=A_i' $ for all $ i\in I $, we ask that $ N(\A')=\phi(N(\A)) $. In particular, the nested set found by our theorem should not depend on the universe into which the family $ \A $ is embedded.

The assumptions of \cref{thm:splinter} are not sufficient to guarantee the existence of such a canonical set.
Consider the example where we have just two separations, $s$ and $t$, which are crossing and let $\A = (A_1) = (\{s,t\})$. Note that $\A$ splinters, but there may be an automorphism that swaps the two separations so the choice of any single one of them is non-canonical.
Since the separations are crossing we cannot use both of them for our nested set either.

For obtaining a canonical nested set, one crucial ingredient will be the notion of {\em extremal} elements of a set of separations, which was already used in~\cite{ProfilesNew}. Given a set $ A\sub U $ of (unoriented) separations, an element $ a\in A $ is {\em extremal} in $ A $, or an {\em extremal element} of $ A $, if $ a $ has some orientation $ \va $ that is a maximal element of $ \vA $. (Recall that $\vA$ is the set of orientations of separations in $A$.) The set of extremal elements of a set of separations is an invariant of separation systems in the following sense: if $ E $ is the set of extremal elements of some set $ A\sub S $ of separations, and $ \phi $ is an isomorphism between $ \vS $ and some other separation system, then $ \phi(E) $ is precisely the set of extremal separations of $ \phi(A) $. Moreover, the extremal separations of a set $ A\sub U $ are nested with each other under relatively weak assumptions: for instance, it suffices that for any two separations in $ A $ at least two of their corner separations also lie in~$ A $.

Let us formally state a set of assumptions under which we can prove a canonical version of \cref{thm:splinter}. Given two separations $ r $ and $ s $ and two of their corner separations $ c_1 $ and $ c_2 $, we say that $ c_1 $ and $ c_2 $ are {\em from different sides of} $ r $ if, for orientations of $ c_1 $, $ r $, and $ s $ with $ \vc_1=(\vr\meet\vs) $, there is an orientation $ \vc_2 $ of $ c_2 $ such that either $ \vc_2=(\rv\meet\vs) $ or $ \vc_2=(\rv\meet\sv) $. Note that $ c_1 $ and $ c_2 $ being from different sides of $ r $ does not imply that $ c_1 $ and $ c_2 $ are distinct separations; consider for instance the edge case that $ r=s=c_1=c_2 $.

Let $ \A=(A_i\mid i\in I) $ be a finite collection of non-empty finite subsets of $ U $ and let $ \curlyle $ be any partial order on $ I $. We write $i \prec j$ if and only if $i\curlyle j$ and $i\neq j$. We say that $ \A $ {\em splinters hierarchically} if for all $ a_i\in A_i $ and $ a_j\in A_j $ the following two conditions hold:

\begin{enumerate}[label={(\arabic*)}]
\item\label{property:hierarchically1_diff} If $ i\prec j $, either some corner separation of $ a_i $ and $ a_j $ lies in $ A_j $, or two corner separations of $ a_i $ and $ a_j $ from different sides of $ a_i $ lie in $ A_i $.

\item\label{property:hierarchically2_equal} If neither $i\prec j$ nor $j\prec i$, there are $ k\in\{i,j\} $ and corner separations $ c_1 $ and $ c_2 $ of $ a_i $ and $ a_j $ from different sides of $ a_k $ such that $ c_1\in A_k $ and $ c_2\in A_i\cup A_j $.
\end{enumerate}

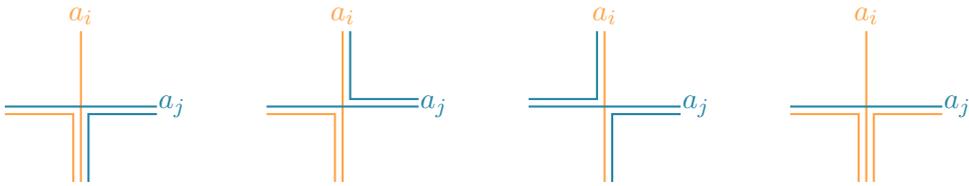
\begin{figure}[h]
    \centering
    \begin{tikzpicture}[scale=.1]
    \draw[thick,colorA] (  0,-10) -- (  0, 10);
    \node[colorA, scale=1] at (0,12) {\clap{$a_i$}};
    \draw[thick,colorB] (-10,  0) -- ( 10,  0);
    \node[colorB, scale=1] at (12, 0) {$a_j$};
    \draw[colorB,thick] (1, -10) -- (1, -1) -- (10, -1);
    \draw[colorA,thick] (-1, -10) -- (-1, -1) -- (-10, -1);
    \end{tikzpicture}\hfil
    \begin{tikzpicture}[scale=.1]
    \draw[thick,colorA] (  0,-10) -- (  0, 10);
    \node[colorA, scale=1] at (0,12) {\clap{$a_i$}};
    \draw[thick,colorB] (-10,  0) -- ( 10,  0);
    \node[colorB, scale=1] at (12, 0) {$a_j$};
    \draw[colorA,thick] (-1, -10) -- (-1, -1) -- (-10, -1);
    \draw[colorB,thick] (1, 10) -- (1, 1) -- (10, 1);
    \end{tikzpicture}\hfil
    \begin{tikzpicture}[scale=.1]
    \draw[thick,colorA] (  0,-10) -- (  0, 10);
    \node[colorA, scale=1] at (0,12) {\clap{$a_i$}};
    \draw[thick,colorB] (-10,  0) -- ( 10,  0);
    \node[colorB, scale=1] at (12, 0) {$a_j$};
    \draw[colorB,thick] (1, -10) -- (1, -1) -- (10, -1);
    \draw[colorB,thick] (-1, 10) -- (-1, 1) -- (-10, 1);
    \end{tikzpicture}\hfil
    \begin{tikzpicture}[scale=.1]
    \draw[thick,colorA] (  0,-10) -- (  0, 10);
    \node[colorA, scale=1] at (0,12) {\clap{$a_i$}};
    \draw[thick,colorB] (-10,  0) -- ( 10,  0);
    \node[colorB, scale=1] at (12, 0) {$a_j$};
    \draw[colorA,thick] (1, -10) -- (1, -1) -- (10, -1);
    \draw[colorA,thick] (-1, -10) -- (-1, -1) -- (-10, -1);
    \end{tikzpicture}
    
    \caption{The possible configurations in \ref{property:hierarchically2_equal} in the definition of \emph{splinter hierarchically}, up to symmetry.}
\end{figure}

In particular if $ \curlyle $ is the trivial partial order on $ I $ in which all $ i\ne j $ are incomparable, then $ \A $ splinters hierarchically if and only if \ref{property:hierarchically2_equal} holds for all $ a_i\in A_i $ and $ a_j\in A_j $; this special case which ignores the partial order on $ I $ is perhaps the cleanest form of an assumption that suffices for a canonical nested set meeting all $ A_i $ in $ \A $. The reason we need to allow a partial order $ \curlyle $ on $ I $ and the slightly weaker condition in \ref{property:hierarchically1_diff} for comparable elements of $ I $ is that otherwise we would not be able to deduce \cref{thm:Profiles_canon} in full from our main theorem of this section due to a quirk in the way that robustness is defined for profiles in~\cite{ProfilesNew} (see \cref{sec:applications_canonical}).

Our first lemma enables us to find a canonical nested set inside $ \bigcup_{i\in I}A_i $ for a collection of sets $ A_i $ whose indexing set is an antichain:

\begin{LEM}\label{lem:ex_nested}
    Let $ (A_i\mid i\in I) $ be a collection of subsets of $ U $ that splinters hierarchically. If $ K\sub I $ is an antichain in $ \curlyle $, then the set of extremal elements of $ \bigcup_{k\in K}A_k $ is nested.
\end{LEM}

\begin{proof}
    Suppose that $ K\sub I $ is an antichain and that for some $ i,j\in K $ there are $ a_i\in A_i $ and $ a_j\in A_j $ such that $ a_i $ and $ a_j $ are extremal in $ \bigcup_{k\in K}A_k $ but cross. Let $ \va_i $ and $ \va_j $ be the orientations of $ a_i $ and $ a_j $ witnessing their extremality. As $ a_i $ and $ a_j $ cross, there are three ways of orienting $a_i$ and $a_j$ such that the supremum of this orientation is strictly larger than $ \va_i $ or $ \va_j $. Hence none of these corner separations can lie in $ A_i\cup A_j $, since that would contradict the maximality of $ \va_i $ or $ \va_j $ in $ \bigcup_{k\in K}\vA_k $. On the other hand, since neither $ i\prec j $ nor $ j\prec i $, by condition~\ref{property:hierarchically2_equal} and the assumption that $ a_i $ and $ a_j $ cross there are at least two orientations of $a_i$ and $a_j$ whose corresponding supremum lies in $ A_i\cup A_j $, causing a contradiction to the extremality of $a_i$ and $a_j$.
\end{proof}

We are now able  to prove a canonical version of the Splinter Lemma by repeatedly applying \cref{lem:ex_nested} to the collection of the $ A_i $ of $ \curlyle $-minimal index that have not yet been met by the nested set constructed so far:

\splinterThmcanonical*

\begin{proof}
    We proceed by induction on $ \abs{I} $. If $ \abs{I}=1 $ we can choose as $ N $ the set of extremal elements of $ A_i $, which is nested by \cref{lem:ex_nested} and clearly canonical.
    
    So suppose that $ \abs{I}>1 $ and that the claim holds for all smaller index sets. Let $ K $ be the set of minimal elements of $ I $ with respect to~$ \curlyle $. By \cref{lem:ex_nested} the set $ E=E(\A) $ of extremal elements of $ \bigcup_{k\in K}A_k $ is nested. Let $ J\sub I $ be the set of indices of all those $ A_j $ that do not meet $ E $, and for $ j\in J $ let $ A_j' $ be the set of all elements of $ A_j $ that are nested with $ E $. We claim that the collection $ \A'=(A_j'\mid j\in J) $ splinters hierarchically with respect to $ \curlyle $ on~$ J $. This follows from \cref{lem:fish} as soon as we show that each $ A_j' $ is non-empty.
    
    To see that each $ A_j' $ is non-empty, for $ j\in J $ let $ a_j $ be an element of $ A_j $ that crosses as few elements of $ E $ as possible. We wish to show that $ a_j $ is nested with $ E $ and thus $ a_j\in A_j' $. So suppose that $ a_j $ crosses some separation in $ E $, that is, some $ a_i\in A_i\cap E $ with $ i\in I\sm J $. Since $ i $ is a minimal element of $ I $ we have either $ i\curlyle j $ or that $ i $ and $ j $ are incomparable. We shall treat these cases separately.
    
    Consider first the case that $ i\prec j $. By condition \ref{property:hierarchically1_diff} of splintering hierarchically, either some corner separation of $ a_i $ and $ a_j $ lies in $ A_j $, or two corner separations of $ a_i $ and $ a_j $ from different sides of $ a_i $ lie in $ A_i $. The first of these possibilities contradicts the choice of $ a_j $, since that corner separation in $ A_j $ would cross fewer elements of $ E $ by \cref{lem:fish}. On the other hand, the latter of these possibilities contradicts the choice of $ a_i $ as an extremal element of $ \bigcup_{k\in K}A_k $. Thus the case $ i\curlyle j $ is impossible.
    
    Let us now consider the case that $ i $ and $ j $ are incomparable. Again, by the choice of $ a_j $, none of the corner separations of $ a_i $ and $ a_j $ can lie in $ A_j $ by \cref{lem:fish}. Therefore condition \ref{property:hierarchically2_equal} of splintering hierarchically yields the existence of a corner separation of $ a_i $ and $ a_j $ in $ A_i $ for each side of $ a_i $; this, however, contradicts the extremality of $ a_i $ in $ \bigcup_{k\in K}A_k $ as before.
    
    Therefore each of the sets $ A_j' $ with $ j\in J $ is non-empty, and hence the collection $ \A'=(A_j'\mid j\in J) $ splinters hierarchically with respect to~$ \curlyle $. Since $ \abs{J}<\abs{I} $ we may apply the induction hypothesis to this collection to obtain a canonical nested set $ N' = N(\A') $ meeting all $ A_j' $. Now $ N=N'\cup E $ is a nested subset of $ U $ which meets every $ A_i $ for $ i\in I $. It remains to show that $ N $ is canonical.
    
    To see that $ N $ is canonical let $ \phi $ be an isomorphism of separation systems between $ \bigcup_{i\in I}\vA_i $ and a subset of some universe $ U' $ such that $ \phi(\A) $ splinters hierarchically with respect to $ \curlyle $ in $ U' $. Then $ \phi(E)=E(\phi(\A)) $, i.e., the set of extremal elements of $ \bigcup_{i\in I}\phi(A_i) $ is exactly $ \phi(E) $. Therefore $ \phi(E) $ meets $ \phi(A_i) $ if and only if $ E $ meets $ A_i $. Consequently the restriction of $ \phi $ to $ \bigcup_{j\in J}\vA_j' $ is an isomorphism of separation systems between $ \bigcup_{j\in J}\vA_j' $ and its image in $ U' $ with the property that $ \phi(\A') $ splinters hierarchically with respect to $ \curlyle $ on $ J $. Moreover, for $ j\in J $, the image $ \phi(A_j') $ of $ A_j' $ is exactly the set of those separations in $ \phi(A_j) $ that are nested with $ \phi(E) $.
    
    Thus we can apply the induction hypothesis to find that $ N(\phi(\A'))=\phi(N(\A')) $. Together with the above observation that $ \phi(E(\A))=E(\phi(\A)) $ this gives
    \[ \phi(N(\A))=\phi(E(\A))\cup\phi(N(\A'))=E(\phi(\A))\cup N(\phi(\A'))=N(\phi(\A)), \]
    concluding the proof.
\end{proof}

\section{Applications of the Canonical Splinter Lemma}\label{sec:applications_canonical}

In this section we apply \cref{thm:splinter_hierarchically} to obtain a short proof of \cref{thm:Profiles_canon}, to strengthen \cref{thm:cliques} for clique separations so as to make it canonical, and finally to establish a canonical tree-of-tangles theorem for another type of separations, so-called circle separations.

\subsection{Robust profiles}\label{sec:appl:profiles_canonical}

Having established \cref{thm:splinter_hierarchically} in the previous section, we are now ready to derive the full version of~\cref{thm:Profiles_canon}. For this let $U={(\vU,\le,{}^*,\join,\meet,|\cdot|)}$ be a submodular universe of separations and $ \P $ a robust set of profiles in $ U $, and let~$ I $ be the set of all pairs of distinguishable profiles in $ \P $. As in \cref{sec:appl:profiles_noncanonical}, for $ \{P,P'\}\in I $ we let
\[ A_{P,P'}:=\menge{a\in U\mid a\tn{ distinguishes }P\tn{ and }P'\tn{ efficiently}}, \]
and let $ \A_\P $ be the family $ ({A_{P,P'}\mid \{P,P'\}\in I)} $. We furthermore define a partial order $ \curlyle $ on $ I $ by letting $\{P,P'\}\prec \{Q,Q'\}$ if and only if the order of some element of $A_{P,P'}$ is strictly lower than the order of some element of $A_{Q,Q'}$. Note that the separations in a fixed $ A_{P,P'} $ all have the same order.

We shall be able to deduce \cref{thm:Profiles_canon} from \cref{thm:splinter_hierarchically} as soon as we show that $ \A_\P $ splinters hierarchically.

\pagebreak[3]
\begin{LEM}\label{lem:profiles_hierarchically}
$\A_\P$ splinters hierarchically with respect to~$ \curlyle $.
\end{LEM}

\begin{proof}
Let $r\in A_{P,P'}$ and $s\in A_{Q,Q'}$ be given. By switching their roles if necessary we may assume that~$|r|\le |s|$. Then $Q$ and $Q'$ both orient $r$; we may assume without loss of generality that $ \vr\in Q $. We will make a case distinction depending on the way $ Q' $ orients~$ r $.

Let us first treat the case that $ Q $ and $ Q' $ orient $r$ differently, i.e., that~$ \rv\in Q' $. Then $ r $ distinguishes $ Q $ and $ Q' $ and hence $ \abs{r}=\abs{s} $ by the efficiency of $ s $. This implies that $ \{P,P'\} $ and $ \{Q,Q'\} $ are either the same pair or else incomparable in~$ \curlyle $. We may assume further without loss of generality that $ \vs\in Q $ and $ \sv\in Q' $. Consider now the two corner separations $ \vr\join\vs $ and $ \vr\meet\vs $: if at least one of these two has order at most $ \abs{s} $, then this corner separation would distinguish $ Q $ and $ Q' $ by the profile property. The efficiency of $ s $ would then imply that this corner separation has order exactly $ \abs{s} $ and hence lies in~$ A_{Q,Q'} $. The submodularity of the order function implies that this is the case for at least one, and therefore for both of these corner separations, yielding the existence of two corner separations of $ r $ and $ s $ from different sides of $ s $ in~$ A_{Q,Q'} $ and showing that \ref{property:hierarchically2_equal} is satisfied.

Let us now consider the case that $ Q $ and $ Q' $ orient $ r $ in the same way, i.e., that~$ \vr\in Q' $. We make a further split depending on whether $ \abs{r}=\abs{s} $ or~$ \abs{r}<\abs{s} $.

Suppose first that $ \abs{r}=\abs{s} $; then neither $ \{P,P'\} \prec \{Q,Q'\} $ nor $\{Q,Q'\}\prec\{P,P'\}$. We may assume that $ P $ and $ P' $ orient $ s $ in the same way: for if $ P $ and $ P' $ orient $ s $ differently, we may switch the roles of $ r $ and $ s $ as well as $ \{P,P'\} $ and $ \{Q,Q'\} $ and apply the above case. So suppose that both of $ P $ and $ P' $ contain $ \vs $, say. Then neither of the corner separations $\vr\join\sv$ nor $\rv\join\vs$ can have order strictly less than $|r|=|s|$, as these corner separations would distinguish $Q$ and $Q'$ or $P$ and $P'$, respectively, and would therefore contradict the efficiency of $s$ or of~$r$, respectively. The submodularity of $ \abs{\cdot} $ now implies that both of these corner separations have order exactly $ \abs{r}=\abs{s} $ and hence lie in $ A_{Q,Q'} $ and $ A_{P,P'} $, respectively, showing that \ref{property:hierarchically2_equal} holds.

Finally, let us suppose that $ \abs{r}<\abs{s} $; then $ \{P,P'\}\prec\{Q,Q'\} $. Consider the two corner separations $ \vr\join\vs $ and $ \vr\join\vs $: if both of $ \vr\join\vs $ and $ \vr\join\vs $ have order strictly greater than $ \abs{s} $, then by the submodularity of the order function both of the other two corner separations $ \rv\join\vs $ and $ \rv\join\sv $ have order strictly smaller than $ \abs{r} $. By the robustness of $ \P $ one of these two corner separations would distinguish $ P $ and $ P' $, contradicting the efficiency of~$ r $.

Thus we may assume at least one of $ \vr\join\vs $ and $ \vr\join\vs $ has order at most $ \abs{s} $. Then that corner separation distinguishes $ Q $ and $ Q' $. In fact, it does so efficiently and hence lies in $ A_{Q,Q'} $, showing that \ref{property:hierarchically1_diff} holds and concluding the proof.
\end{proof}

We are now ready to deduce the full \cref{thm:Profiles_canon} from \cref{thm:splinter_hierarchically}:

\profilesCanonToT*

\begin{proof}
By \cref{lem:profiles_hierarchically} the family $ \A_\P $ splinters hierarchically. Thus we can apply \cref{thm:splinter_hierarchically} to $ \A_\P $ to obtain a nested set $ N=N(\A_\P) $ which meets every~$ A_{P,P'} $. Clearly, $ N $ satisfies (i), (ii) and (iv) of \cref{thm:Profiles_canon}.

To see that $ N $ satisfies (iii), let $ \alpha $ be an automorphism of $ \vU $. Then the restriction of $ \alpha $ to $ \bigcup_{\{P,P'\}\in I}\vA_{P,P'} $ is an isomorphism of separation systems onto its image in $ \vU $. We therefore have, by \cref{thm:splinter_hierarchically}, that $ \alpha(N(\A_\P))=N(\alpha(\A_\P)) $. For every $ A_{P,P'} $ in $ \A_\P $ we have that $ \alpha(A_{P,P'}) $ is precisely the set of those separations in $ U $ which distinguish $ P^\alpha $ and $ P'^\alpha $ efficiently; in other words, we have $ \alpha(\A_\P)=\A_{\P^\alpha} $, showing that (iii) is satisfied.
\end{proof}

\subsection{Clique separations}\label{sec:appl:cliques_canonical}

Regarding the profiles of clique separations discussed in  \cref{sec:appl:cliques_noncanonical}, \cref{lem:clique_corners} not only suffices to show that the sets $A_{P,P'}$ splinters, but can be used to show that the collection of these $A_{P,P'}$ even splinters hierarchically, allowing us to apply~\cref{thm:splinter_hierarchically}: for this we simply define the same partial order $\curlyle$ on the set of pairs $\{P,P'\}$ as in the previous section, that is, $\{P,P'\} \prec \{Q,Q'\}$ if and only if $|r| < |s|$ for some (equivalently: for all) $ r\in A_{P,P'} $ and~$ s\in A_{Q,Q'} $.

To see this, let $ P,P' $ and $ Q,Q' $ be distinguishable pairs of profiles of clique separations. Let $r\in A_{P,P'} $ and $ s\in A_{Q,Q'}$, and suppose without loss of generality that $|r|\le |s|$. If $r$ and $s$ are nested, then $ r $ and $ s $ themselves are corner separations of $ r $ and $ s $ that lie in $ A_{P,P'} $ and $ A_{Q,Q'} $, respectively. However, if $r$ and $s$ cross, then by \cref{lem:clique_corners} there are orientations of $r$ and $s$ such that ${\abs{\rv\meet\sv},\abs{\rv\meet\vs}\le\abs{r}}$ and ${\abs{\rv\meet\sv},\abs{\rv\meet\vs},\abs{\vr\meet\sv}\le\abs{s}}$. By switching their roles if necessary we may assume that $ \vr\in P $ and $ \rv\in P' $, and likewise that $ \vs\in Q $ and~$ \sv\in Q' $.

Since $(\vr\meet\sv),(\rv\meet\sv) \le \sv$ and $ \sv\in Q' $, the profile $ Q' $ contains both of these corner separations by consistency. On the other hand, by the assumption that $ \abs{r}\le\abs{s} $, the separation $ r $ gets oriented by $ Q $, and consequently by the profile property $ Q $ must contain the inverse of one of those two corner separations. This corner separation then distinguishes $ Q $ and $ Q' $, and in fact it does so efficiently, since its order is at most $ \abs{s} $, meaning that this corner separation lies in $ A_{Q,Q'} $. Therefore, if $ \abs{r}<\abs{s} $, condition~\ref{property:hierarchically1_diff} of splintering hierarchically is satisfied.

So suppose further that $ \abs{r}=\abs{s} $, and let us check that condition~\ref{property:hierarchically2_equal} of splintering hierarchically is satisfied. Observe that, similarly as above, $ P $ orients $ s $, and $ P' $ contains both $ (\vr\meet\sv) $ and $ (\rv\meet\vs) $ by consistency with $ \rv\in P' $, implying as before that one of $(\rv\meet\sv)$ and $(\rv\meet\vs)$ also efficiently distinguishes $P$ and $P'$, i.e., is an element of $A_{P,P'}$. If this corner separation in $ A_{P,P'} $ and the corner separation in $ A_{Q,Q'} $ found above are from different sides of either $ r $ or $ s $, then condition~\ref{property:hierarchically2_equal} of splintering hierarchically would be satisfied. So suppose not; that is, suppose that $ (\rv\meet\sv) $ distinguishes both $ P $ and $ P' $ as well as $ Q $ and $ Q' $ efficiently. In particular $ \abs{\rv\meet\sv}=\abs{r}=\abs{s} $, and hence by the last part of~\cref{lem:clique_corners}, all four corner separations of $ r $ and $ s $ have order at most $ \abs{r} $. Consequently, since $ P' $ orients $ s $, one of $ (\vr\meet\sv) $ and $ (\vr\meet\vs) $ distinguishes $ P $ and $ P' $ efficiently, which one depending on whether $ \vs\in P' $ or $ \sv\in P' $. In either case we have found a corner separation of $ r $ and $ s $ in $ A_{P,P'} $, which together with $ (\rv\meet\sv)\in A_{Q,Q'} $ witnesses that~\ref{property:hierarchically2_equal} is fulfilled.

Therefore, by \cref{thm:splinter_hierarchically} we get that we can choose the set in \cref{thm:cliques} canonically:

\begin{THM}\label{thm:cliques_canon}
    For every set $\P$ of profiles of clique separations of a graph $G$, there is a nested set $N=N(\P)$ of separations which efficiently distinguishes all the distinguishable profiles in $\P$ and is canonical, that is, such that $N(\P^\alpha)= N(\P)^\alpha$ for every automorphism $\alpha$ of the underlying graph $G$.
\end{THM}

\begin{proof}
Every automorphism of $G$ induces an automorphism of the separation system. Hence we can obtain the claimed nested set by applying \cref{thm:splinter_hierarchically} to the family of the sets $A_{P,P'}$ of those clique separations which efficiently distinguish the pair $ P,P' $ of distinguishable profiles in $\P$.
\end{proof}

\subsection{Circle separations}\label{sec:appl:circle_separations}

Another special case of separation systems are those of {\em circle separations} discussed in~\cite{AbstractTangles}: given a fixed cyclic order on a ground-set $ V $, a {\em circle separation} of $ V $ is a bipartition $ (A,B) $ of $ V $ into two disjoint intervals in the cyclic order. Observe that the set of all circle separations is not closed under joins and meets and hence not a sub-universe of the universe of all bipartitions of $ V $:

\begin{EX}\label{ex:circles_not_universe}
    Consider the natural cyclic order on the set $V=\menge{1,2,3,4}$. The bipartitions $(\menge{1}, \menge{2,3,4})$ and $(\menge{3}, \menge{4,1,2})$ of $ V $ are circle separations. However, their supremum in the universe of all bipartitions of $ V $ is $ (\menge{1,3},\menge{2,4}) $, which is not a circle separation.
\end{EX}

Let $ V $ be a ground-set with a fixed cyclic order and ${U=(\vU,\le,{}^*,\join,\meet,|\cdot|)}$ the universe of all bipartitions of $ V $ with a submodular order function $ |\cdot| $. Let $ S\sub U $ be the set of all separations in $ U $ that are circle separations of $ V $. Consequently we denote by $ S_k $ the set of all those circle separations in $ S $ whose order is $ {<k} $.

Given fixed integers $ m \ge 1 $ and $ n > 3 $, we call a consistent orientation of $ S_k $ a {\em $ k $-tangle} in $ S $ if it has no subset in
\[ \F = \F_m^n\coloneqq\menge{F\sub 2^\vU \;\bigg|\; \big\lvert{\textstyle  \bigcap_{(A,B)\in F}B}\big\rvert <m\text{ and }\abs{F}<n}. \]
A {\em tangle} in $ S $ is then a $ k $-tangle for some $ k $, and a {\em maximal tangle} in $ S $ is a tangle not contained in any other tangle in $ S $. As usual, two tangles are {\em distinguishable} if neither of them is a subset of the other; a separation $ s $ {\em distinguishes} two tangles if they orient $ s $ differently, and $ s $ does so {\em efficiently} if it is of minimal order among all separations in $ S $ distinguishing that pair of tangles.

Using~\cref{thm:splinter_hierarchically} we can show that there is a canonical nested set of circle separations which efficiently distinguishes all distinguishable tangles in $ S $:

\begin{THM}\label{thm:circles}
    The set $S$ of all circle separations of $V$ contains a tree set $T=T(S)$ that efficiently distinguishes all distinguishable tangles of $S$. Moreover, this tree set $T$ can be chosen canonically, i.e., so that for every automorphism $\alpha$ of $S$ we have $T(S^\alpha)=T(S)^\alpha$.  
\end{THM}

In order to prove \cref{thm:circles} we need the following short lemma:

\begin{LEM}\label{lem:circle_corners}
Let $r$ and $s$ be two circle separations of $ V $. If $r$ and $s$ cross, then all four corner separations of $r$ and $s$ are again circle separations.
\end{LEM}

\begin{proof}
Let $\vr=(A,B)$ and $ \vs=(C,D)$. Since $r$ and $s$ cross, the sets $A\cap C$ and $B\cap D$ are non-empty and moreover intervals in the cyclic order. Thus $B\cup D$ is also an interval and therefore $\vr\meet\vs=(A \cap C\,,\,B \cup D)$ is indeed a circle separation. 
\end{proof}

Let us now prove \cref{thm:circles}.

\begin{proof}[Proof of \cref{thm:circles}]
For every pair $  P,P' $ of distinguishable tangles in $ S $ let $A_{ P, P'}$ be the set of all circle separations that efficiently distinguish $ P$ and $ P'$. We define a partial order~$ \curlyle $ on the set of all pairs of distinguishable tangles by letting $\menge{ P, P'}\prec\menge{ Q, Q'}$ for two distinct such pairs if and only if the separations in $A_{ P, P'}$ have strictly lower order than those in~$A_{ Q, Q'}$.

Let us show that the collection of these sets $ A_{ P, P'} $ splinters hierarchically; the claim will then follow from \cref{thm:splinter_hierarchically}.

For this let $ P,P'$  and $  Q, Q'$ be two distinguishable pairs of tangles in $ S $ and let $ r\in A_{ P, P'}$ and $ s\in A_{ Q, Q'}$. If $r$ and $s$ are nested, then $ r $ and $ s $ themselves are corner separations from different sides of $ r $ and $ s $ that lie in $ A_{ P, P'} $ and $ A_{ Q, Q'} $, respectively, in which case there is nothing to show.

So suppose that $ r $ and $ s $ cross. Then by \cref{lem:circle_corners} all corner separations of $ r $ and $ s $ are circle separations. By switching their roles if necessary we may assume that~$|r|\le|s|$; we shall treat the cases of $ |r|<|s|$ and $ |r|=|s| $ separately.

Let us first consider the case that $|r|<|s|$. Then $ \{ P, P'\}\prec\{ Q, Q'\} $, so it suffices to show that \ref{property:hierarchically1_diff} is satisfied, i.e., to find a corner separation of $ r $ and $ s $ in $ A_{ Q, Q'} $. Since $ Q$ and $ Q'$ both orient $s$, which is of higher order than $ r $, both $  Q $ and $  Q' $ also orient $r$. By $ |r|<|s| $ and the efficiency of $ s $, $ r $ cannot distinguish $  Q $ and $  Q' $. Thus some orientation $ \vr $ of $ r $ lies in both $  Q $ and $  Q' $.

By renaming them if necessary we may assume that $ \vr\in P $ and $ \rv\in P' $. Suppose now that one of $ \vr\join\vs $ and $ \vr\join\sv $ has order at most $ |s| $. Then $  Q $ and $  Q' $ would both orient that corner separation, and they would do so differently by the definition of a tangle. Thus that corner separation would lie in $ A_{ Q, Q'} $, as desired.

Hence we may assume that both of $ \vr\join\vs $ and $ \vr\join\sv $ have order higher than $ |s| $. Then, by submodularity, both $ \vr\meet\vs $ and $ \vr\meet\sv $ have order less than $ |r| $. Therefore both of these corner separations get oriented by $  P $ and $  P' $, but neither of them can distinguish $  P $ and $  P' $ by the efficiency of $ r $. In fact by the consistency of $  P $ and $  P' $ we must have $ (\vr\meet\vs),(\vr\meet\sv)\in P\cap P' $. However the set $ \{\rv,\,(\vr\meet\vs),\,(\vr\meet\sv)\} $ lies in $ \F $, contradicting the assumption that $  P $ and $  P' $ are tangles in $ S $.

It remains to deal with the case that $ |r|=|s| $ and show that \ref{property:hierarchically2_equal} is satisfied. For this we shall find corner separations from different sides of $ r $ or of $ s $ that lie in $ A_{ P, P'} $ and $ A_{ Q, Q'} $, respectively. By the submodularity of the order function, and by switching the roles of $ r $ and $ s $ if necessary, we may assume that there are orientations of $ r $ and $ s $ such that both $ \vr\join\vs $ and $ \vr\join\sv $ have order at most $ |r| $. By possibly renaming $ \vs $ and $ \sv $ we may further assume that $ \vr\join\vs $ distinguishes $  P $ and $  P' $. Then, by the efficiency of $ r $, we must have $ |\vr\join\vs|=|r| $, and hence $ |\rv\join\sv|\le|s| $ by submodularity. Recall that we assumed $ |\vr\join\sv|=|r|=|s| $, so one of $ \rv\join\sv $ and $ \vr\join\sv $ must distinguish $  Q $ and $  Q' $. Again, that corner separation must in fact distinguish $  Q $ and $  Q' $ efficiently, i.e., lie in $ A_{ Q, Q'} $. Now this corner separation together with $ \vr\join\vs $ witnesses that \ref{property:hierarchically2_equal} holds.
\end{proof}


\vspace{1cm}
\noindent
\begin{minipage}{\linewidth}
 \raggedright\small
   \textbf{Christian Elbracht},
   Universität Hamburg,
   Hamburg, Germany \\
   \texttt{christian.elbracht@uni-hamburg.de}

   \textbf{Jakob Kneip},
   Universität Hamburg,
   Hamburg, Germany \\
   \texttt{jakob.kneip@uni-hamburg.de}
   
   \textbf{Maximilian Teegen},
   Universität Hamburg,
   Hamburg, Germany \\
   \texttt{maximilian.teegen@uni-hamburg.de}
\end{minipage}

\end{document}